\documentclass{amsart}
\usepackage{amssymb,latexsym}
\theoremstyle{plain}

\newtheorem{corollary}{Corollary}
\newtheorem{proposition}{Proposition}
\newtheorem{lemma}{Lemma}
\theoremstyle{definition}

\newtheorem{remark}{Remark}

\newcommand{\enm}[1]{\ensuremath{#1}}          %
\newcommand{\cal}[1]{\mathcal{#1}}

\newcommand{\Nm}{\mathrm{Num}}

\newcommand{\CC}{\enm{\mathbb{C}}}
\newcommand{\II}{\enm{\mathbb{I}}}
\newcommand{\NN}{\enm{\mathbb{N}}}
\newcommand{\RR}{\enm{\mathbb{R}}}

\newcommand{\ZZ}{\enm{\mathbb{Z}}}
\newcommand{\FF}{\enm{\mathbb{F}}}

\newcommand{\PP}{\enm{\mathbb{P}}}

\newcommand{\Bb}{\enm{\cal{B}}}
\newcommand{\Cc}{\enm{\cal{C}}}

\renewcommand{\phi}{\varphi}
\renewcommand{\theta}{\vartheta}
\renewcommand{\epsilon}{\varepsilon}

\renewcommand{\to}[1][]{\xrightarrow{\ #1\ }}

\newcommand{\old}[1]{}

\date{}

\begin{document}

\title[numerical ranges]
{The Hermitian null-range of a matrix over a finite field}
\author{E. Ballico}
\address{Dept. of Mathematics\\
 University of Trento\\
3123 Povo (TN), Italy}
\email{ballico@science.unitn.it}
\thanks{The author was partially supported by MIUR and GNSAGA of INdAM (Italy).}

\begin{abstract}
Let $q$ be a prime power. For $u=(u_1,\dots ,u_n), v=(v_1,\dots ,v_n)\in \FF _{q^2}^n$ let $\langle u,v\rangle := \sum _{i=1}^{n} u_i^qv_i$ be the Hermitian form of $\FF _{q^2}^n$. Fix an $n\times n$ matrix $M$ over $\FF _{q^2}$.
We study the case $k=0$ of the set $\Nm _k(M):= \{\langle u,Mu\rangle \mid u\in \FF _{q^2}, \langle u,u\rangle  =k\}$. When $M$ has coefficients in $\FF _q$ we study the set $\Nm _0(M)_q:= \{\langle u,Mu\rangle \mid u\in \FF _q^n\}\subseteq \FF _q$. The set $\Nm _1(M)$ is the numerical range of $M$, previously introduced
in a paper by Coons, Jenkins, Knowles, Luke and Rault (case $q$ a prime $p\equiv 3\pmod{4}$) and by myself (arbitrary $q$). We study in details $\Nm _0(M)$ and $\Nm _0(M)_q$ when $n=2$.
If $q$ is even, $\Nm _0(M)_q$ is easily described for arbitrary $n$.
\end{abstract}

\maketitle

\section{Introduction}
Fix a prime $p$ and a power $q$ of $p$. Up to field isomorphisms there is a unique field $\FF _q$ such that $\sharp (\FF _q) =q$ (\cite[Theorem 2.5]{ln}).
Let $e_1,\dots ,e_n$ be the standard basis of $\FF _{q^2}^n$. For all $v,w\in \FF _{q^2}^n$, say $v = a_1e_1+\cdots +a_ne_n$ and $w =b_1e_1+\cdots +b_ne_n$, set $\langle v,w\rangle
= \sum _{i=1}^{n} a_i^qb_i$. $\langle \ ,\ \rangle$ is the standard Hermitian form of  $\FF_{q^2}^n$. The set $\{u\in \FF _{q^2}^n\mid \langle u,u\rangle =1\}$ is an affine chart of the Hermitian variety
of $\PP^n(\FF _{q^2})$ (\cite[Ch. 5]{h}, \cite[Ch. 23]{ht}). Let $M$ be an $n\times n$ matrix with coefficients in $\FF _{q^2}$. In \cite{b} we made the following definition.
The \emph{numerical range} $\Nm (M)$ (or $\Nm _1(M)$) of $M$ is the set of
all $\langle u,Mu\rangle$ with $\langle u,u\rangle =1$. $\CC $ is a degree $2$ Galois extension of $\RR$ with the complex conjugation as the generator of the Galois group.
$\FF _{q^2}$ is a degree $2$ Galois extension of $\FF _q$ with the map $t\mapsto t^q$ as a generator of the Galois group. Hence $\langle \ ,\ \rangle$ is the Hermitian form
associated to this Galois extension. Thus the definition of $\Nm (M)$ is a natural extension of the notion of  numerical range in linear algebra (\cite{gr}, \cite{hj}, \cite{hj1}, \cite{pt}). This extension was introduced in \cite{cjklr} when $q$
is a prime $p\equiv 3\pmod{4}$. In this paper we consider related subsets $\Nm '_0(M)\subseteq \Nm _0(M)\subseteq \FF _{q^2}$.

As in \cite{cjklr} for any $k\in \FF _q$ set $C_n(k):= \{(a_1,\dots ,a_n)\in \FF _{q^2}^n\mid \sum _{i=1}^{n} a_i^{q+1} =k\}$. The set $C_n(0)$ is a cone of $\FF _{q^2}^n$ and
its proiectivization $\Cc _n\subset \PP^{n-1}(\FF _{q^2})$ is the Hermitian variety of dimension $n-2$ of $\PP^{n-1}(\FF _{q^2})$ with rank $n$. Set $C'_n(0): =C_n(0)\setminus \{0\}$. Recall that
$\langle u, u\rangle \in \FF _q$ for all $u\in \FF _{q^2}^{n}$. 
For any $n\times n$ matrix over $\FF _{q^2}$ and any $k\in \FF _q$ let $\Nm _k(M)$ (resp $\Nm '_0(M)$) be the set of all $a\in \FF_{q^2}$ such that there is $u\in C_n(k)$  (resp. $u\in C'_n(0)$ and $n\ge 2$)
with $a = \langle u,Mu\rangle$. We always have $0\in \Nm _0(M)$, $\Nm _0(M) =\Nm '_0(M)\cup \{0\}$ and quite often, but not always, we have $0\in \Nm '_0(M)$ (Propositions \ref{ab3}, \ref{ab4}, \ref{ab5}). For
instance, we have $\Nm '_0(\II _{n\times n})=\{0\}$ for all $n\ge 2$. If $n=1$, i.e. $M$ is the multiplication by a scalar $m$, we have $\Nm _k(M) = mk$. There is an ambiguity if $n=1$, because $C'_1(0) =\emptyset$. Hence we do not define $\Nm '_0$ for $1\times 1$ matrices. We say that $\Nm '_0(M)$ is the {\emph{Hermitian null-range}} of the matrix $M$.

We have $\Nm _k(M) = k\Nm _1(M)$ for all $k\in \FF _q^\ast$ (use Remark \ref{+01} to adapt the proof \cite[Lemma 2.3]{cjklr}). Thus we know all numerical ranges of $M$ if we know $\Nm _1(M)$ and $\Nm '_0(M)$. The first part of this paper
studies $\Nm '_0(M)$. If $n=2$ we prove several results concerning the set $\Nm '_0(M)$ under different assumptions on the eigenvalues and the eigenvectors of $M$. As a byproduct of our study of the case $n=2$ we get the following result.

\begin{corollary}\label{ab6}
Assume that $M\ne c\II _{n\times n}$ for some $c$. Then $\sharp (\Nm _0(M)) \ge \lceil (q+1)/2\rceil$.
\end{corollary}

In the second part of this paper we consider the following question. Fix $k\in \FF _q$ and suppose that all coefficient $m_{ij}$ of the matrix $M$ are elements of $\FF _q$.
For any $k\in \FF _q$ let $\Nm _k(M)_q$ be the set of all $a\in \FF _q$ such that
there is $u\in \FF _q^n$ with $\langle u,u\rangle = k$ and $\langle u,Mu\rangle =a$. If $n>1$, $k=0$ and we also impose that $u\ne 0$, then we get
the definition of $\Nm '_0(M)_q$. Note that $\Nm _k(M)_q \subseteq \Nm _k(M)\cap \FF _q$ and that $\Nm '_0(M)_q \subseteq \Nm' _0(M)\cap \FF _q$. These inclusions are not always equalities (see for instance part (i) of Proposition \ref{qq3}). In this part there are huge difference between the case $q$ even and the case $q$ odd.

In the case $q$ even, for any matrix $M$ we have $\Nm '_0(M)_q\ne \emptyset$, either $\Nm '_0(M)_q =\{0\}$ or $\Nm '_0(M)_q \supseteq\FF _q^\ast$,  and $\Nm '_0(M) =\{0\}$ if and only if $m_{ij}+m_{ji} =0$ for all $i\ne j$ (see Proposition \ref{qq4} for a more general result).

In the case $q$ odd, there is a difference between the case $q\equiv 1\pmod{4}$ (in which $-1$ is a square in $\FF _q$) and the case $q\equiv -1\pmod{4}$ (in which
$-1$ is a not square in $\FF _q$). For instance if $n=2$ and $q\equiv -1\pmod{4}$, then $\Nm '_0(M)_q=\emptyset$ (part (i) of Proposition \ref{qq3}). Now assume $n=2$
and $q\equiv 1\pmod{4}$. By part (iii) of Proposition \ref{qq3} we have:

\begin{enumerate}
\item If $m_{12}+m_{21}\ne 0$, then $\Nm _0(M)_q$ contains at least $(q-1)/2$ elements of $\FF _q^\ast$.

\item Assume $m_{12}+m_{21} =0$. If $m_{11} =m_{22}$, then $\Nm _k(M)_q =\{km_{11}\}$ for all $k\in \FF _q$ and $0\in \Nm '_0(M)_q$. If $m_{11}\ne m_{22}$,
then $\sharp (\Nm _k(M)_q)\le (q+1)/2$ for all $k\in \FF_q$, $\sharp (\Nm _0(M)_q) =(q+1)/2$ and $\sharp (\Nm '_0(M)_q) =(q-1)/2$.\end{enumerate}

\section{Preliminaries}

Let $\II _{n\times n}$ denote the unity $n\times n$ matrix. For any $n\times n$ matrix $N = (n_{ij})$, $n_{ij}\in \FF _{q^2}$ for all $i,j$, set $N^{\dagger} = (n_{ji}^q)$. For all $u,v\in \FF _{q^2}^n$ we have $\langle u,Nv\rangle = \langle N^{\dagger}u,v\rangle$.
The matrix $N$ is called unitary if $N^{\dagger} N = \II _{n\times n}$ (or equivalently $NN^{\dagger} =\II _{n\times n}$). Note that $\Nm _k(M) = \Nm _k(U^{\dagger}MU)$ for every unitary matrix $U$.

\begin{remark}\label{100}
Fix a prime $p$ and let $r$ be a power of $p$. Up to field isomorphisms there is a unique finite field, $\FF _r$, with $r$ elements and $\FF _r =\{x\in \overline{\FF _p}\mid x^r=x\}$.
The group $\FF _r^\ast$ is a cyclic group of order $r-1$ and $\FF _r^\ast = \{x\in \overline{\FF _p}\mid x^{r-1} =1\}$ (\cite[page 1]{h}, \cite[Theorem 2.8]{ln}).
\end{remark}

\begin{remark}\label{+01}
Fix $a\in \FF _q^\ast$. Since $q+1$ is invertible in $\FF _q$, the polynomial $t^{q+1}-a$ and its derivative $(q+1)t^q$ have no common zero. Hence the polynomial $t^{q+1}-a$ has $q+1$ distinct roots in $\overline{\FF _q}$. Fix any one of them, $b$.
Since $a^{q-1} =1$ (Remark \ref{100}), we have $b^{q^2-1} =1$. Hence $b\in \FF _{q^2}^\ast$. Thus there are exactly $q+1$ elements $c\in \FF _{q^2}^\ast$ with $c^{q+1} =a$.
\end{remark}

\begin{remark}\label{0ab1}
Let $\FF$ be a finite field. If $\FF$ has even characteristic, then for each $a\in \FF$ there is a unique $b\in \FF$ with $b^2=a$ (e.g. because $\FF ^\ast$ is a cyclic group with odd order by Remark \ref{100}). If $\FF$ has odd characteristic, then each element of $\FF$ is
a sum of $2$ squares of elements of $\FF$ (\cite[Lemma 5.1.4]{h}).
\end{remark}

\begin{remark}\label{aab1}
If $n\ge 2$, then $\Nm '_0(\II _{n\times n})=\{0\}$, because $C_n(0)\ne \{0\}$ for all $n\ge 2$.
\end{remark}

\begin{lemma}\label{aab2}
$\sharp (\Nm _0(M)) = \sharp (\Nm _0(M^{\dagger}))$ and $\sharp (\Nm'  _0(M)) = \sharp (\Nm'_0(M^{\dagger}))$.
\end{lemma}

\begin{proof}
Fix $u\in \FF _{q^2}^n$ and let $M$ be an $n\times n$ matrix over $\FF _{q^2}$. We have $\langle u,Mu\rangle = \langle M^{\dagger}u,u\rangle = (\langle u,M^{\dagger}u\rangle )^q$.
 Since $\FF _{q^2}^\ast$ is a cyclic group
of order $(q+1)(q-1)$ and $q$ is coprime with $(q+1)(q-1)$, the map $t\mapsto t^q$ induces a bijection $\FF _{q^2}\to \FF _{q^2}$, proving the lemma.
\end{proof}

\begin{remark}\label{ab2}
Fix $c, d\in \FF _{q^2}$ and $k\in \FF _q$. For any $n\times n$ matrix $M$ over $\FF _{q^2}$ we have $\Nm _k(c\II _{n\times n} +dM) = ck^2 +d\Nm_k(M)$. 
\end{remark}

\begin{lemma}\label{ab1}
Assume $n\ge 2$ and that $M = A\oplus B$ (orthonormal decomposition) with $A$ an $x\times x$ matrix, $B$ an $(n-x)\times (n-x)$ matrix and $0<x<n$. Then  $\Nm _0(M) =\Nm _0(A)+\Nm _0(B) \cup \bigcup _{k\in \FF _q^\ast}(k(\Nm _1(A)-\Nm _1(B))$. We have $0\in \Nm '_0(M)$ if and only if either $x\ge 2$ and $0\in \Nm '_0(A)$ or $x\le n-2$ and $0\in \Nm '_0(B)$ or there is $a\in \Nm _1(A)$ with $-a\in \Nm _1(B)$.
\end{lemma}
\begin{proof}
Take $u = (v,w)\in \FF _{q^2}^n$ with $\langle u,u\rangle =0$, $v\in \FF _{q^2}^x$ and $w\in \FF _{q^2}^{n-x}$. We have $\langle u,Mu\rangle = \langle v,Av\rangle
+ \langle v,Bv\rangle$. We have $\langle u,u\rangle =\langle v,v\rangle +\langle w,w\rangle$
and hence the assumption ``$\langle u,u\rangle =0$'' is equivalent to the assumption ``$\langle w,w\rangle =-\langle v,v\rangle$'' (note that this is also true when $q$ is even).
First assume $\langle v,v\rangle =0$. We get $\langle w,w\rangle =0$, $\langle v,Av\rangle \in \Nm _0(A)$ and  $\langle w,Aw\rangle \in \Nm _0(B)$
and so $\Nm _0(M)\supseteq \Nm _0(A)+\Nm _0(B)$. Now assume $k:= \langle v,v\rangle \ne 0$. We get $\langle u,Mu\rangle =a+b$ with $a\in \Nm _k(A)$ and
$b\in \Nm _{-k}(B)$. Since $\Nm _x(M) =x\Nm _1(M)$ for all $x\ne 0$, we have $\Nm _k(M) = -\Nm _{-k}(M)$ if $k\ne 0$. Hence $\Nm _0(M)\subseteq \Nm _0(A)+\Nm _0(B) \cup \bigcup _{k\in \FF _q^\ast}k(\Nm _1(A)-\Nm _1(B))$. Since $u=0$ if and only if $v=0$ and $w=0$,
we get that $0\in \Nm '_0(M)$ if and only if we came from a case with $k\ne 0$ or with a case in which $\langle v,v\rangle =\langle w,w\rangle =0$
and either $v\ne 0$ or $w\ne 0$.
\end{proof}

\begin{proposition}\label{ab3}
Assume that $M$ is unitarily equivalent to a diagonal matrix with $c_1,\dots ,c_k$, $k\ge 2$, different eigenvalues and $c_i$ occurring with multiplicity $x_i>0$.

\quad (a) If $k\ge 3$, then $\Nm _0(M) =\FF _{q^2}$.

\quad (b) If $k\ge 3$, then $0\in \Nm '_0(M)$ if and only if either $k\ge 4$ or $n\ge 4$ or $n=k=3$ and $(c_3-c_1)/(c_2-c_1)\in \FF _q^\ast$.

\quad ({c}) If $k=2$ and $n\ge 3$, then $\Nm '_0(M) =\{t(c_2-c_1)\}_{t\in \FF _q}$.

\quad (d) If $k=n=2$, then  $\Nm '_0(M) =\{t(c_2-c_1)\}_{t\in \FF _q^\ast}$.
\end{proposition}

\begin{proof}
Note that $c_i-c_j\in \FF _{q^2}^\ast$ for all $i\ne j$. Since $\FF _{q^2}$ is a $2$-dimensional $\FF _q$-vector space, $c_3-c_1$ and $c_2-c_1$
are a basis of $\FF _{q^2}$ over $\FF _q$ (i.e. $(c_3-c_1)/(c_2-c_1)\in \FF _q^\ast$) if and only if $c_3-c_2$ and $c_1-c_2$ are another basis of $\FF _{q^2}$. Hence $(c_3-c_1)/(c_2-c_1)\in \FF _q^\ast$ $\Leftrightarrow$ $(c_3-c_2)/(c_1-c_2)\in \FF _q^\ast$ $\Leftrightarrow$ $(c_2-c_1)/(c_3-c_1)\in \FF _q^\ast$. 

By Remark \ref{ab1} we reduce to the case $c_1=0$. Fix $a\in \FF _{q^2}$.

\quad (i) First assume $k\ge  3$. Up to a unitary transformation we may assume that $e_1$ is an eigenvector of $M$ with eigenvalue $0$, $e_2$ is an eigenvector of $M$ with
eigenvalue $c_2\in \FF _{q^2}\setminus \{0\}$ and $e_3$ is an eigenvector of $M$ with
eigenvalue $c_3\in \FF _{q^2}\setminus \{0,c_2\}$. Since $\FF _{q^2}$ is a two-dimensional $\FF_q$-vector space, there are uniquely determined $a_2,a_3\in \FF _q$
such that $a = a_2c_2+a_3c_3$. By Remark \ref{+01} there are $u_i\in \FF _{q^2}$, $i=2,3$, such that $u_i^{q+1} = a_i$, $i=2,3$. Take $u_1\in \FF _{q^2}$ such that $u_1^{q+1} = -a_2-a_3$
(Remark \ref{+01}) and set $u:= u_1e_1+u_2e_2+u_3e_3$. We have $\langle u,u\rangle = \sum _{i=1}^{3} u_i^{q+1} =0$ and $\langle u,Mu\rangle = c_2u_2^{q+1}+c_3u_3^{q+1} =a$.
Hence $\Nm _0(M) =\FF _{q^2}$, proving part (a). 

\quad (ii) Now take $k=n= 3$. We need to check when $0\in \Nm '_0(M)$. We need to find $u_1,u_2,u_3\in \FF _{q^2}$ such that
$(u_1,u_2,u_3)\ne (0,0,0)$, $u_1^{q+1} +u_2^{q+1} +u_3^{q+1} =0$ and
$c_1u_1^{q+1} +c_2u_2^{q+1} +c_3u_3^{q+1} =0$. The previous conditions are satisfied if and only if there is $ (u_2,u_3)\ne (0,0)$ such that $(c_2-c_1)u_2^{q+1} +(c_3-c_1)u_3^{q+1}=0$. Since $u_2^{q+1}$ and $u_3^{q+1}$ are elements of $\FF _q$, $c_3-c_2\ne 0$ and $c_2-c_1\ne 0$, this is possible if and
only if $(c_3-c_1)/(c_2-c_1)\in \FF _q$. 

\quad (iii) Now assume $k\ge 4$. We may assume $c_1=0$ and that $e_i$ is an eigenvalue for $c_i$. We get that it is sufficient to find $u_2,u_3,u_4$ with $(u_2,u_3,u_4)\ne (0,0,0)$
and $\sum _{i=2}^{4} (c_i-c_1)u_i^{q+1} =0$. Since the map $\FF _{q^2}^\ast \to \FF _ q^\ast$ defined by the formula $t\mapsto t^{q+1}$ is surjective (Remark \ref{+01}), it is sufficient
to find $b_i\in \FF _q$, $2\le i\le 4$, such that $(b_2,b_3,b_4)\ne (0,0,0)$ and
\begin{equation}\label{eqa+b1}
\sum _{i=2}^{4} (c_i-c_1)b_i =0
\end{equation}
Since $\FF _{q^2}$ is a $2$-dimensional vector space over $\FF _q$, (\ref{eqa+b1}) is equivalent to  a homogenous linear system with $2$ equations and $3$ unknowns over $\FF _q$
and hence it has a non-trivial solution.

\quad (iv) Now assume $k=3$ and $n\ge 4$. Without losing generality we may assume that the eigenspace of $c_1$ contains $e_1,e_2$. Use Remark \ref{aab1}.

\quad (v) Assume $k=2$. We reduce to the case $c_1=0$ and hence $c_2-c_1\ne 0$. Let $V_1$ (resp. $V_2)$ the eigenspace for the eigenvalue $0$ (resp. $c_2-c_1$). Take $u\in \FF _{q^2}$
and write $u = u_1+u_2$ with $u_1\in V_1$ and $u_2\in V_2$. Since $\langle v,w\rangle =0$ for all $v\in V_1$ and $w\in V_2$, we have
$\langle u,u\rangle = \langle u_1,u_1\rangle +\langle u_2,u_2\rangle$ and $\langle u,Mu\rangle = (c_2-c_1)\langle u_2,u_2\rangle$. Since $\langle u_2,u_2\rangle \in \FF _q$,
we get $\Nm _0(M)\subseteq \{t(c_2-c_1)\}_{t\in \FF _q}$. Since we may take as $\langle u_2,u_2\rangle$
any $\alpha \in \FF _q$ and then take $u_1$ with $\langle u_1,u_1\rangle =-\alpha$, we get $\Nm _0(M) = \{t(c_2-c_1)\}_{t\in \FF _q}$. If $n=2$ we have
$\langle u,Mu\rangle =0$ if and only if $u_2=0$. Hence it $n=2$ we have $\langle u,u\rangle =0$
if and only if $u_1=u_2=0$ and so $0\notin \Nm '_0(M)$. If $n\ge 3$, then $x_i\ge 2$ for some $i$ and hence $0\in \Nm '_0(M)$ (Remark \ref{aab1}).
\end{proof}

The case $a=-1$ of Remark \ref{+01} gives the following lemma.

\begin{lemma}\label{ab7.3}
Set $\Theta := \{a\in \overline{\FF _q}\mid a^{q+1}=-1\}$. Then $\sharp (\Theta )=q+1$ and $\Theta \subset \FF _{q^2}^\ast$.
\end{lemma}

We write $M = (m_{ij})$.

\begin{proposition}\label{ab4}
Take $n=2$ and assume that $M$ has a unique eigenvalue, $c$, and that the associated eigenspace is one-dimensional and generated by an eigenvector $u$ with $\langle u,u\rangle \ne 0$.
We have $0\notin \Nm '_0(M)$. If $q$ is even, then $\Nm '_0(M) =\FF _{q^2}^\ast$. If $q$ is odd, then $\sharp (\Nm '_0(M)) =(q^2-1)/2$.
\end{proposition}

\begin{proof}
Taking $M-c\II _{2\times 2}$ instead of $M$ we reduced to the case $c=0$. Take $t\in \FF _{q^2}$ such that $t^{q+1}= \langle u,u\rangle$ (Remark \ref{+01}). Using $t^{-1}u$ instead of $u$ we reduce
to the case $\langle u,u\rangle =1$. Hence, up to a unitary transformation we reduce to the case $u=e_1$. In this case we have $m_{11} =m_{21}=0$.
Since $m_{22}$ is an eigenvalue of $M$, we have $m_{22}=0$. Since $e_2$ is not an eigenvector of $M$, we have $m_{12}\ne 0$. Take
$v = ae_1+be_2$ such that $\langle v,v\rangle =0$, i.e. such that $a^{q+1}+b^{q+1} =0$. We have $\langle v,Mv\rangle = \langle v,m_{12}be_1\rangle = a^qbm_{12}$.
Note that $a=0$ if and only if $b=0$ and hence $0\notin \Nm'_0(M)$. Take $\Theta$ as in Lemma \ref{ab7.3}. Since the multiplication by $m_{12}$
is injective, it is sufficient to count the number of elements of the set $\Delta$ of all $a^qb$ with $ab\ne 0$ and $a^{q+1}+b^{q+1} =0$.
There is a unique $z\in \Theta$ such that $b=az$, but for a fixed $a$ we may take any $z\in \Theta$ and then set $b:= az$. Varying $a\in \FF _{q^2}^\ast$ we get as $a^{q+1}$ all elements of $\FF _q^\ast$ (Remark \ref{+01}). Thus $\Delta$ is the set
of all products $cz$ with $c\in \FF _q^\ast$ and $z\in \Theta$.  Note that $\sharp (\FF _q^\ast )\cdot \sharp (\Theta ) =\sharp (\FF _{q^2}^\ast )$ by Lemma \ref{ab7.3}. Take $c, c_1\in \FF _q^\ast$ and $z, z'\in \Theta$ and assume
$cz = c_1z_1$. Hence $c^{q+1}z^{q+1} = c_1^{q+1}z_1^{q+1}$. Since $z^{q+1} =z_1^{q+1} =-1$, we get $c^{q+1} =c_1^{q+1}$. Since $c, c_1\in \FF _q^\ast$,
we get $c^2=c_1^2$. If $q$ is even, we get $c=c_1$. Hence $z=z_1$. Hence if $q$ is even we get $\sharp (\Nm '_0(M)) =q^2-1$ and (since $0\notin \Nm '_0(M)$), we get
$\Nm '_0(M) =\FF _{q^2}^\ast$. Now assume that $q$ is odd. We get that either $c=c_1$ or $c=-c_1$. If $c=c_1$, then we get $z=z_1$. Now assume $c=-c_1$ and hence
$z=-z_1$. We get $cz = (-c)(-z)$. In this case the set of all $cz$, $c\in \FF _q^\ast$ and $z\in \Theta$ has cardinality $(q^2-1)/2$ and hence $\sharp (\Nm '_0(M)) =(q^2-1)/2$.
\end{proof}

\begin{proposition}\label{ab5}
Take $n=2$ and assume that $M$ has two distinct eigenvalues $c_1,c_2$ and eigenvectors $u_i$ of $c_i$, $1\le i \le 2$, with $\langle u_i,u_i\rangle =0$ for all $i$.
Then there is $o\in \FF _{q^2}^\ast$ such that $\Nm '_0(M) =\{to\}_{t\in \FF _q}$.
\end{proposition}

\begin{proof}
Each $u_i$ gives that $0\in \Nm '_0(M)$. Since $u_1$ and $u_2$ are a basis of $\FF _{q^2}^2$, $\langle \ ,\ \rangle$ is non-degenerate and $\langle u_i,u_i\rangle =0$ for all $i$, we have
$e:= \langle u_1,u_2\rangle \ne 0$. Note that $\langle u_2,u_1\rangle =e^q$. Taking $M-c_1\II_{2\times 2}$ instead of $M$ we reduce to the case $c_1=0$ and hence $c:= c_2-c_1 \ne 0$. Take
$a, b\in \FF _{q^2}^\ast$ and set $u:= au_1+bu_2$. We have $\langle u,u\rangle =e^qb^qa +ea^q$. Hence $\langle u,u\rangle =0$ if and only if $e^qb^qa +ea^b=0$.
We have $\langle u,Mu\rangle = \langle u,cbu_2\rangle = ea^qbc$. Set $w:= eb/a$. We have $\langle u,u\rangle =0$ if and only if $w^q+w=0$. Since $b\ne 0$, we have $w\ne 0$ and so
$\langle u,u\rangle =0$ if and only if $w^{q-1} +1=0$.
We have $\langle u,Mu\rangle =a^{q+1}wc$. By Remark \ref{+01} varying $a\in \FF_{q^2}^\ast$ we get as $a^{q+1}$ an arbitrary element of $\FF _q^\ast$. If $q$ is even, $w$ is an arbitrary element of $\FF _q^\ast$, because $w^{q-1}=1$ and $\FF _q^\ast =\{t\in \overline{\FF _q}\mid t^{q-1}=1\}$
and hence varying $a$ and $w$ we get that $\Nm '_0(M) =\{tc\}_{t\in \FF _q}$. Now assume that $q$ is odd. In this case $w\notin \FF _q$, because $w^{q-1} =-1\ne 1$
(Remark \ref{100}). Take $w_1\in \FF _{q^2}$ with $w_1^{q-1} =-1$ (Remark \ref{+01}).
Since $(w/w_1)^{q-1} =1$, we have $w/w_1\in \FF _q^\ast$. Hence varying $w$ with $w^{q-1}=1$ and $a^{q+1}$ with $a\in \FF _{q^2}^\ast$ we get
exactly $q-1$ elements of $\FF _{q^2}^\ast$, all of them of the form $\{to\}_{t\in \FF _q^\ast}$ with $o=wc$. 
\end{proof}

\begin{proposition}\label{ab7.2}
Take $n=2$ and assume $m_{21}\ne 0$ and $m_{12}\ne 0$. Then:

\quad (i) $\sharp (\Nm '_0(M)) \ge \lceil (q+1)/2\rceil$;

\quad (ii) If $(-m_{12}/m_{21})^{q+1} \ne 1$, then $\sharp (\Nm '_0(M)) \ge q+1$.
\end{proposition}

\begin{proof}
 Using $M-m_{11}\II _{2\times 2}$ instead of $M$ we reduce to the case $m_{11}=0$ (Remark \ref{ab2}). Take $u = ae_1+be_2$. We have $\langle u,u\rangle = a^{q+1}+b^{q+1}$, $Mu = bm_{21}e_1 +(am_{12}+m_{22}b)e_2$ and $\langle u,Mu\rangle
= a^qbm_{21} + b^q(am_{12}+m_{22}b) = a^qbm_{21} +b^qam_{12} + m_{22}b^{q+1}$. We take only the solutions obtained taking $b=1$ and so $a\in \Theta$, where
$\Theta$ is as in Lemma \ref{ab7.3}. To get the lemma we study the number of different values of the restriction to $\Theta$ of the polynomial $g(t) = m_{21}t^q + m_{12}t+m_{22}$.
This number is the number of different values of the restriction to $\Theta$ of the polynomial $f(t) = m_{21}t^q + m_{12}t$. Fix $z, w\in \Theta$ and assume $f(z)=f(w)$.
Hence $f(z)zw = f(w)zw$. Since $z^{q+1} = w^{q+1} = -1$, we get $-m_{21}w+ m_{12}z^2w = -m_{21}z+m_{12}zw^2$. Set $h_z(t) = m_{12}zt^2-m_{12}z^2t+ m_{21}t -m_{21}z$.
The polynomial $h_z(t)$ has at most two zeroes in $\FF _{q^2}$, one of them being $z$. Hence for each $z\in \Theta$ there is at most one $w\in \Theta$ with $w\ne z$
and $g(w) =g(z)$. Thus $\sharp (\Nm '_0(M)) \ge \lceil (q+1)/2\rceil$. Assume the existence of $w\ne z$ with $h_z(w) =0$. Since $z$ and $w$ are the two roots of $h_z(t)$,
we have $m_{12}z^2w = -m_{21}z$, i.e. (since $z\ne 0$, $m_{12}zw = -m_{21}$. Since $(zw)^{q+1} = 1$ and $(-1)^{q+1} =1$ (even if $q$ is even), we get part (ii).
\end{proof}

\begin{proof}[Proof of Corollary \ref{ab6}:]
It is sufficient to the case $n=2$. Using $M-m_{11}\II _{2\times 2}$ we reduce to the case $m_{11}=0$ (Remark \ref{ab2}). If $m_{21}=0$, then we use either Lemma \ref{ab3} (if $M$ is unitarily equivalent to a diagonal matrix)
or Lemma \ref{ab5} (if $M$ is diagonalizable, but not unitarily diagonalizable) or Lemma \ref{ab4} (if $0$ is the unique eigenvalue of $M$ with $e_1$ as its eigenvalue).
If $m_{21} =0$ we apply the last sentence to $M^{\dagger}$ and use Lemma \ref{aab2}.  Hence we may assume that $m_{12}m_{21} \ne 0$. Apply Proposition \ref{ab7.2}.\end{proof}

\section{Matrices with coefficients in $\FF_q$}\label{Sq}

We always assume $n\ge 2$. We assume $M =(m_{ij})$ with $m_{ij}\in \FF _q$ for all $i, j$.
Take $k\in \FF _q$ and $u \in \FF _q^n$  with $\langle u,u\rangle = k$ and write $u = \sum _{i=1}^{n} x_ie_i$ with $x_i\in \FF _q$ for all $i$. Since $x_i\in \FF_q$, we have
$x_i^{q+1} =x_i^2$ and so the condition $\langle u,u\rangle = k$ is equivalent to the degree $2$ equation
\begin{equation}\label{eqq1}
\sum _{i=1}^{n} x_i^2 = k
\end{equation}
Since $x_i^q =x_i$ for all $i$, the condition $\langle u,Mu\rangle =a$ is equivalent to
\begin{equation}\label{eqq2}
\sum _{i,j=1}^{n} m_{ij} x_ix_ j =a
\end{equation}

\begin{remark}\label{00+a+1}
Fix any $k\in \FF _q$, any integer $n\ge 2$ and any $n\times n$ matrix $M$ with coefficients in $\FF_q$. Every element of $\FF _q$ is a sum of two squares
of elements of $\FF _q$ (Remark \ref{0ab1}). Hence (\ref{eqq1}) has always a solution $(y_1,\dots ,y_n)\in \FF_q^n$. Setting $x_i:= y_i$ in the left hand side of (\ref{eqq2}) we get $\Nm _k(M)_q\ne \emptyset$. However, there are a few cases with $\Nm '_0(M)_q=\emptyset$ (Proposition \ref{qq3}). We always have $\Nm '_0(M)_q\ne \emptyset$ if $q$ is even (Proposition
\ref{qq4}).
\end{remark}

Set $\Bb _n:= \{u\in \FF_q^n\mid \langle u,u\rangle =0\}$. Let $\nu '_M: \Bb _n\to \FF _q$ be the map defined by the formula $\nu '_M(u) =\langle u,Mu\rangle$.

\begin{remark}\label{qq1}
Take another $n\times n$ matrix $N =(n_{ij})$ with coefficient in $\FF _q$, with $n_{ii} =m_{ii}$ for all $i$ and $n_{ij}+n_{ji} = m_{ij}+m_{ji}$ for all $i\ne j$. The systems given by
(\ref{eqq1}) and (\ref{eqq2}) for $M$ and for $N$ are the same and hence $\Nm _k(M)_q = \Nm _k(N)_q$ for all $k$ and $\Nm '_0(M)_q = \Nm '_0(N)_q$. As a matrix $N$
we may always take a triangular matrix. If $q$ is odd (i.e. if we may divide by $2$ in our fields $\FF_q$ and $\FF _{q^2}$), then we may take as $N$ a symmetric matrix. Take
any symmetric matrix $N$ with coefficients in $\FF _q$. Since the coefficients of $N$ are
in $\FF _q$ (i.e. $n_{hk}^q =n_{hk}$ for all $h, k$) and $N$ is symmetric, then $N^{\dagger} = N$ and so $\Nm '_0(N)$ is invariant by the Frobenius map $t\mapsto t^q$.
$\Nm '_0(N)_q$ is the set of all fixed points for the Frobenius action on $\Nm '_0(N)$.\end{remark}

\begin{remark}\label{qq2}
For all $c, d\in \FF _q$ we have $\Nm  _0(c\II _{n\times n}+dM)_q = d\Nm '_0( M)_q$ and $\Nm _k(c\II _{n\times n}+dM)_q = ck^2 + d\Nm _k(M)_q$.
\end{remark}

\begin{remark}\label{0ab2}
Fix $k, b \in \FF _q^\ast$, $a\in \FF _q$, and assume the existence of $d\in \FF _q^\ast$ such that $b =kd^2$. The map $(x_1,\dots ,x_n)\to (dx_1,\dots ,dx_n)$
shows that the system given by (\ref{eqq1}) and (\ref{eqq2}) has a solution if and only the system gives by (\ref{eqq1}) and (\ref{eqq2}) with $b$ instead of $k$ and
$ad^2$ instead of $a$ has a solution. Hence $\sharp (\Nm _k(M)_q) = \sharp (\Nm _b(M)_q)$. If $q$ is even, for all $k, b \in \FF _q^\ast$, $a\in \FF _q$ there is $d\in \FF _q^\ast$ such that $b =kd^2$ (Remark \ref{0ab1}). Hence if $q$ is even, then $\sharp (\Nm _k(M)_q) = \sharp (\Nm _1(M)_q)$ for all $k\in \FF _q^\ast$.
Now assume $q$ odd. The multiplication group $\FF _q^\ast$ is cyclic of order $q-1$ (Remark \ref{100}). Since $q-1$ is even, the group $\FF _q^\ast /(\FF _q^\ast )^2$ has cardinality $2$
and hence to know all integers $\sharp (\Nm _k(M)_q)$, $k\in \FF_q^\ast$, it is sufficient to know it for one $k$, which is a square in $\FF _q^\ast$ (e.g. for $k=1$) and
for one $k$, which is not a square in $\FF _q^\ast$.
\end{remark}

\quad (a) Assume that $q$ is even. For any $k\in \FF _q$ there is a unique $c\in \FF _q$ with $c^2=k$ (Remark \ref{0ab1}). Hence (\ref{eqq1}) is equivalent to $(\sum _{i=1}^{n} x_i +c)^2 =0$,
i.e. to 
\begin{equation}\label{eqq3}
\sum _{i=1}^{n} x_i =c.
\end{equation}
Hence the system given by (\ref{eqq1}) and (\ref{eqq2}) is equivalent to the system given by (\ref{eqq2}) and (\ref{eqq3}). Writing $x_n =\sum_{i=1}^{n-1} x_i+c$ we translate
the system given by (\ref{eqq2}) and (\ref{eqq3}) into a degree $2$ polynomial in $x_1,\dots ,x_{n-1}$. If $k=a=0$, then this is a homogeneous polynomial of degree $2$ in $n-1$ variables and hence it has a non-trivial solution
if $n-1\ge 3$ (\cite[Corollary 1]{h}, \cite[Theorem 3.1]{s}), proving the following result.

\begin{corollary}\label{0abq3}
If $M$ has coefficients in $\FF_q$, $q$ is even and $n\ge 4$, then $0\in \Nm '_0(M)_q$.
\end{corollary}

If $k$ and/or $a$ are arbitrary the system  given by (\ref{eqq2}) and (\ref{eqq3}) is equivalent to find a solution in $\FF _q^{n-1}$ of a certain polynomial in $ \FF _q[x_1,\dots ,x_{n-1}]$ with degree at most $2$. We only fix
$c\in \FF _q$, but not $a$. Call
$f(x_1,\dots ,x_{n-1})$ the left hand side of (\ref{eqq2}) obtaining substituting $x_n = -x_1-\cdots -x_n +c$. $\Nm _k(M)_q$ is described by the image of the map $\FF _q^{n-1} \to \FF_q$
associated to the polynomial $f(x_1,\dots ,x_{n-1})$ with $\deg (f)\le 2$. We claim that if $f$ is not a constant polynomial, then the image of $f$ has cardinality at least $q/2$. Indeed,
if $\deg (f)=1$, then $f$ induces a surjective map $\FF _q^{n-1}\to \FF _q$. Now assume $\deg (f) =2$. For any map $h: \FF _q\to \FF _q$ induced by a degree $2$ polynomial
a fiber of $h$ has cardinality at most $2$. Hence $\sharp (h(\FF _q))\ge q/2$. Hence $\sharp (f(\FF _q^{n-1}) \ge q/2$. See part (ii) of Proposition \ref{qq3} for a case with $f\equiv 0$, $\Nm _0(M)_q= \{0\}$ and $\Nm '_0(M)_q=\emptyset$.

\quad (b) Assume that $q$ is odd. Taking $a=k=0$, we get that (\ref{eqq1}) and (\ref{eqq2}) are a system of two degree $2$ homogeneous equations. Chevalley-Warning theorem
(\cite[Theorem 3.1]{s}) gives the following corollary.

\begin{corollary}\label{0abq2}
If $M$ has coefficients in $\FF_q$, $q$ is odd and $n\ge 5$, then $0\in \Nm '_0(M)_q$.
\end{corollary}

The left hand side of (\ref{eqq1}) is a non-degenerate quadratic form $\beta\in \FF_q[x_1,\dots ,x_n]$. If $n =2s$ $\beta$ is characterized in \cite[Table 5.1]{h} with
$m= n$ (because all the coefficients, $1$, appearing on the left hand side of (\ref{eqq1}) are squares in $\FF _q$): it is a hyperbolic quadric if either $s$ is even
or
$q \equiv 1\pmod{4}$ and $s$ is odd, while it is elliptic if $s$ is odd and $q\equiv -1\pmod{4}$.

Now we consider the case $n=2$ for an arbitrary $q$.

\begin{proposition}\label{qq3}
Assume $n=2$ and let $N=(n_{ij})$ be the $2\times 2$-matrix with $n_{11}=m_{11}$, $n_{22} = m_{22}$, $n_{21}=0$ and $n_{12} = m_{12}+m_{21}$.
We have $\Nm '_0(M)_q = \Nm '_0(N)_q$ and $\Nm _k(M)_q = \Nm  _k(N)_q$ for all $k\in \FF _q$. 

\quad (i) If $q\equiv -1\pmod{4}$, then $\Nm '_0(M)_q=\emptyset$. 

\quad (ii) Assume that $q$ is even. If $m_{22}+m_{12}+m_{21} +m_{11} \ne 0$, then $\Nm '_0(M)_q =\FF _q^\ast$ and $\sharp (\Nm _k(M)_q) \ge q/2$ for all $k\in \FF _q^\ast$. If $m_{22}+m_{12}+m_{21}+m_{11}=0$,
then $\Nm '_0(M)_q =\{0\}$ and for any fixed $k\in \FF _q^\ast$ either $\Nm _k(M) =\FF _q$ or $\sharp (\Nm _k(M)_q)=1$. If $m_{12}+m_{21}=0$ and $m_{11} \ne m_{22}$,
then $\Nm _k(M)_q =\FF _q$ for all $k\in \FF _q^\ast$.

\quad (iii) Assume $q\equiv 1 \pmod{4}$. 

\quad (iii1) If $m_{12}+m_{21}\ne 0$, then $\Nm _0(M)_q$ contains at least $(q-1)/2$ elements of $\FF _q^\ast$.

\quad (iii2) Assume $m_{12}+m_{21} =0$. If $m_{11} =m_{22}$, then $\Nm _k(M)_q =\{km_{11}\}$ for all $k\in \FF _q$ and $0\in \Nm '_0(M)_q$. If $m_{11}\ne m_{22}$,
then $\sharp (\Nm _k(M)_q)\le (q+1)/2$ for all $k\in \FF_q$, $\sharp (\Nm _0(M)_q) =(q+1)/2$ and $\sharp (\Nm '_0(M)_q) =(q-1)/2$.
\end{proposition}

\begin{proof}
We have $\Nm _k(N)_q =\Nm _k(M)_q$ and $\Nm '_0(N)_q =\Nm "_0(M)_q$ by Remark \ref{0ab2}.

Take $u = x_1e_1+x_2e_2$ with $\langle u,u\rangle = k$ and $\langle u,Mu\rangle =a$. Hence we get the system given by (\ref{eqq1}) and (\ref{eqq2}). If $q$ is even, then instead
of (\ref{eqq1}) we may use (\ref{eqq3}) with $c^2=k$.

\quad (a) Assume for the moment $q\equiv -1 \pmod{4}$. Thus $q$ is odd and $(-1)^{(q-1)/2} =-1$ in $\ZZ$. Since $\FF _q^\ast$ is a cyclic group of order $q-1$, we get that $-1$ is not a square in $\FF _q^\ast$. Hence (\ref{eqq1}) for $k=0$ has only the solution $x_1=x_2=0$.

\quad (b) Now assume that $q$ is even. Take $k=0$ in (\ref{eqq3}). We have $x_1+x_2=0$ if and only if $x_1=x_2$. When $x_1=x_2$, (\ref{eqq2}) is equivalent to
$(m_{22}+m_{12}+m_{21} +m_{11})x_1^2 =a$. If $m_{22}+m_{12}+m_{21} +m_{11}=0$, then we get $a=0$ and so $\Nm _0(M) =\{0\}$; taking $x_1=x_2=1$ we get $\Nm '_0(M)=\{0\}$. Now assume $m_{22}+m_{12}+m_{21} +m_{11}\ne 0$. If $a=0$, we get $x_1=0$ and so $x_2=0$ and hence $0\notin \Nm '_0(M)$. Now assume $a\ne 0$. There is a unique $b\in \FF _q^\ast$ such that $b^2=a/(m_{22}+m_{12}+m_{21} +m_{11})$ (Remark \ref{0ab1}). Taking $x_1=x_2=b$ we get $a\in \Nm '_0(M)$. 

Now we fix $k\in \FF _q^\ast$ and
write $c^2=k$ with $c\in \FF _q^\ast$ (Remark \ref{0ab1}). We have $x_2=x_1+c$ by (\ref{eqq3}). Substituting this equation in (\ref{eqq2}) we get an equation $f(x_1) =a$
with $\deg (f)\le 2$. The coefficient of $x_1^2$ in $f$ is $m_{11}+m_{12}+m_{22}+m_{21}$. If $m_{11}+m_{12}+m_{22}+m_{21}\ne 0 $, then $\sharp (f(\FF _q)) \ge q/2$, because $\sharp (f^{-1}(t)) \le 2$
for all $t\in \FF _q$. If $m_{11}+m_{12}+m_{22}+m_{21}=0$, then either $f$ has degree $1$ and so it induces a bijection $\FF _q\to \FF _q$ or it is a constant, $\alpha$ (we allow the case $\alpha =0$)
and hence $\Nm _k(M)_q=\{\alpha\}$. Now assume $m_{12}+m_{21}=0$ and $m_{11}\ne m_{22}$. Take $k=c^2$. Substituting (\ref{eqq3}), i.e. $x_2=x_1+c$ in (\ref{eqq2}) 
we get $(m_{11}+m_{22})x_1^2 + c(m_{11}+m_{22}) = a$. Since $m_{11}+m_{22}\ne 0$ and every element of $\FF _q$ is  square (Remark \ref{0ab1}), we get $\Nm _k(M)_q=\FF _q$ for
all $k$.

\quad ({c}) Now assume that $q \equiv 1\pmod{4}$. Since $q\equiv 1\pmod{4}$, then  $(q-1)/2\in \NN$. Since $\FF _q^\ast$ is a cyclic group of order $q-1$, there is $e\in \FF _q^\ast$ with $e^2=-1$. We have $e\ne -e$ and $t^2=-1$ with $t\in \overline{\FF _q}$
if and only if $t\in \{-e,e\}$.  First  take $k=0$ and hence $x_1=tx_2$ with $t^2=-1$, i.e. $t\in \{e,-e\}$. Assume for
the moment $m_{12}+m_{12}\ne 0$. Hence there is $g\in \{e,-e\}$ such that $-m_{11} + g(m_{12}+m_{21}) +m_{22} \ne 0$. Take $x_1 =gx_2$. Since $g^2=-1$, we have $x_1^2+x_2^2=0$
and (\ref{eqq2}) is transformed into $(-m_{11} +g(m_{12}+m_{21}) +m_{22})x_2^2 =a$. Since $(q-1)/2$ elements of $\FF _q^\ast$ are squares (Remark \ref{0ab1})
we get that $\Nm _0(M)_q$ contains at least $(q-1)/2$ elements of $\FF _q^\ast$. 

Now assume $m_{12}+m_{21} =0$. We have $\Nm '_0(M)_q=\Nm _0(N)_q$ and $\Nm _k(M)_q= \Nm _k(N)_q$, where $N =(n_{ij})$ is the diagonal matrix with
$n_{11}=m_{11}$ and $n_{22} =m_{22}$. If $m_{11}=m_{22}$, then $N = m_{11}\II _{2\times 2}$ and hence $\Nm _k(N)_q =\{km_{11}\}$ for all $k\in \FF _q$ and
$0\in \Nm '_0(N)$, because $\nu '(e,1) =0$. Now assume $m_{11} \ne m_{22}$. We fix $k\in \FF_q$, but not $a$. Subtracting $m_{11}$ times (\ref{eqq1}) from (\ref{eqq2}) we get $(m_{22}-m_{11})x_2^2 = a-km_{11}$. 
Since $m_{22}\ne m_{11}$ and $(q+1)/2$ elements of $\FF _q$ are squares, we get that $\sharp (\Nm _k(N)) \le (q+1)/2$ (we only get the inequality $\le $, because for a given $b\in \FF _q$, we are not sure that
the equation $x_1^2 +b^2 =k$ has a solution). If $k=0$, we may always take $x_1=eb$ and so $\sharp (\Nm _0(N)_q) =(q+1)/2$. We have $0\notin \Nm _0'(N)_q$, because
we first get $x_2=0$ and then $x_1=0$.\end{proof}

 The case $k\ne 0$ of step ({c}) of the proof of Proposition \ref{qq3} proves the following observation.

\begin{remark}
Assume $n=2$, $q\equiv 1\pmod{4}$ and $m_{12}+m_{21} =0$.  If $m_{11} =m_{22}$, then $\Nm _k(M)_q =\{km_{11}\}$ for all $k\in \FF _q$. If $m_{11}\ne m_{22}$,
then $\sharp (\Nm _k(M)_q)\le (q+1)/2$ for all $k\in \FF_q^\ast $.
\end{remark}

\begin{corollary}\label{qq5}
Assume $n\ge 2$, $q\equiv 1\pmod{4}$ and fix an $n\times n$-matrix $M= (m_{ij})$ with coefficients in $\FF _q$. 

\quad (i) Assume  $m_{ij}+m_{ji}=0$ for all $i,j$ with $1\le i<j\le n$ and $m_{ii}=m_{11}$ for all $i$. Then $\Nm _k(M)_q= \{km_{11}\}$ for
all $k\in \FF _q$ and $0\in \Nm '_0(M)$.

\quad (ii) If $M$ is not as in (i), then $\Nm _0(M)$ contains  at least $(q-1)/2$ elements of $\FF _q^\ast$.
\end{corollary}

\begin{proof}
Let $N$ be the $n\times n$-matrix with $n_{ii}=m_{ii}$ for all $i$, $n_{ij}=0$ for all $i<j$ and $n_{ij} =m_{ij}+m_{ji}$ for all $i<j$. We have  $\Nm _k(M)_q =\Nm _k(N)_q$ and $\Nm '_0(M)_q = \Nm '_0(N)$ by Remark \ref{qq1}. Take $M$ as in part (i). We have $N = m_{11}\II _{n\times n}$. Hence $\Nm _k(M)_q= \{km_{11}\}$ for
all $k\in \FF _q$. We have $0\in \Nm '_0(N)$, because the equation $x_1^2+x_2^2 =0$ has a non-trivial solution, e.g. $(e,1)$ with $e^2=-1$.
Now assume that $M$ is not as in (i). Hence either there are $i<j$ with $m_{ij}+m_{ji} \ne 0$ or there is $i>1$ with $m_{ii} \ne m_{11}$. In the former (resp. latter) case
we use part (iii1) (resp. (iii2)) of Proposition \ref{qq3}.
\end{proof}

\begin{proposition}\label{qq4}
Assume $n \ge 2$ and $q$ even and fix an $n\times n$-matrix $M= (m_{ij})$ with coefficients in $\FF _q$.

\quad (a) We have $\Nm '_0(M)_q \ne \emptyset$ and either $0\in \Nm '_0(M)$ or $\Nm _0(M)\supseteq \FF _q^\ast$.

\quad (b) We have $\Nm '_0(M)_q =\{0\}$ if and only if $m_{ii} +m_{ij} +m_{ji} +m_{jj} =0$ for all $i<j$.

\quad ({c}) Assume $\Nm '_0(M)_q\ne \{0\}$. If $n=2$, (resp. $n=3$, resp. $n\ge 4$), then $\Nm '_0(M)_q= \FF _q^\ast$ (resp. $\Nm '_0(M)_q \supseteq \FF _q^\ast$, resp. $\Nm '_0(M)_q =\FF _q$).
\end{proposition}

\begin{proof}
Part (a) follows from the case $n=2$, which is true by part (ii) of Proposition \ref{qq3}. 

The ``~only if~'' part of part (b) follows from part (a) and the case $n=2$, which is true by part (ii) of Proposition \ref{qq3}.

Now assume $n\ge 3$ and $m_{ii} +m_{ij} +m_{ji} +m_{jj} =0$ for all $i<j$. Take $u =\sum _{i=1}^{n} x_ie_i$, $x_i\in \FF _q$. For $i=1,\dots ,n$
the coefficient of $x_i^2$ in $\langle u,Mu\rangle$ is $m_{ii}$. If $1\le i <j \le n$ the coefficient of $x_ix_j$ in $\langle u,Mu\rangle$ is $m_{ij}+m_{ji}$.
Now assume $\langle u,u\rangle =0$, i.e.
$x_n=x_1+\cdots +x_{n-1}$. Note that $x_n^2= x_1^2+\cdots +x_{n-1}^2$. Fix $i\in \{1,\dots ,n-1\}$. After this substitution the coefficient of $x_i^2$ in $\langle u,Mu\rangle$ is $m_{ii}+m_{nn} + m_{in}+m_{ni} =0$. Fix $1\le i<j\le n-1$. After the substitution $x_n=x_1+\cdots +x_{n-1}$ the coefficient of $x_ix_j$ in $\langle u,Mu\rangle$ is
$m_{ij}+m_{ji} +m_{ni} +m_{in} +m_{nj} +m_{jn}$. By assumption we have
$m_{ij} +m_{ji} =m_{ii}+m_{jj}$, $m_{ni}+m_{in} =m_{ii}+m_{nn}$ and $m_{nj} +m_{jn} = m_{jj}+m_{nn}$. Hence  $m_{ij}+m_{ji} +m_{ni} +m_{in} +m_{nj} +m_{jn} =2m_{ii}+2m_{jj}+2m_{nn} =0$. Part (a) gives $\Nm '_0(M)_q=\{0\}$.

The case $n=2$ of part ({c}) is true by part (ii) of Proposition \ref{qq3}. Part ({c}) for $n=3$ follows from part (a). Part ({c}) for $n\ge 4$ follows from part (a) and Corollary \ref{0abq3}.
\end{proof}

\begin{proposition}\label{qq2.1}
Fix $c\in \FF _q^\ast$ and set $M:= c\II_{n\times n}$. 

\quad (i) If $q$ is even, then $\Nm '_0(c\II _{n\times n} ) =\{0\}$ for all $n\ge 2$ and $\sharp({\nu '_M}^{-1}(0)) = q^{n-1}$. 

\quad (ii) Assume that $q$ is odd. We have  $\Nm '_0(c\II _{n\times n} ) =\{0\}$ if either $n\ge 3$ or $n=2$ and $q\equiv 1\pmod{4}$, while  $\Nm '_0(c\II _{n\times n} ) =\emptyset$ if $q\equiv -1\pmod{4}$. If $n =2s+1$ is odd, then $\sharp({\nu '_M}^{-1}(0)) =q^{2s}$. If $n=2s$ with either $s$ even or $q\equiv 1\pmod{4}$, then $\sharp({\nu '_M}^{-1}(0)) =q^{2s-1}+q^s-q^{s-1}$. If $n=2s$ with $s$ odd and $q\equiv -1\pmod{4}$, then $\sharp({\nu '_M}^{-1}(0)) = q^{2s-1}-q^s+q^{s-1}$.
\end{proposition}

\begin{proof}
We obviously have $\langle u,c\II _{n\times n}u\rangle =0$ for any $u\in \FF _q$ with $\langle u,u\rangle =0$. Thus the only problem is
if there is $u\in \FF _q^n$, $u\ne 0$, with $\langle u,u\rangle =0$ and to compute the cardinality of the set of all such $u$. Write $u =\sum _i x_ie_i$ with $x_i\in \FF_q$. First assume that $q$ is even. In this case the
condition $\langle u,u\rangle =0$ is equivalent to (\ref{eqq3}) with $c=0$ and it has a non-trivial solution for all $n\ge 2$; moreover the set $\langle u,u\rangle =0$ is the hyperplane $x_1+\cdots +x_n=0$ of $\FF _q^n$ and hence it has cardinality $q^{n-1}$. Now assume that $q$ is odd. In this
case (\ref{eqq1}) with $k=0$ is the equation of a certain quadric hypersurface $Q \subset \PP^{n-1}(\FF _q)$ and $0\in \Nm '_0(c\II _{n\times n})$ if and only if $Q(\FF _q) \ne \emptyset$, while (since we are working in the vector space $\FF _q^n$, instead of the associated projective space)
$\sharp(\nu _M '^{-1}(0)) = 1 +(q-1)\sharp (Q)$. The quadric $Q$ has always full rank and hence $Q\ne \emptyset$ if $n-1\ge 2$. The integer $\sharp (Q)$ is
computed in \cite[Table 5.1 and Theorem 5.2.6]{h}.\end{proof}

\begin{proposition}\label{qq6}
Assume $q\equiv -1 \pmod{4}$ and $n\ge 3$. Then $\Nm '_0(M)\ne \emptyset$.
\end{proposition}

\begin{proof}
It is sufficient to do the case $n=3$. Just use that $x_1^2+x_2^2+x_3^2=0$ has a solution $\ne (0,0,0)$ in $\FF _q^3$ (since $q$ is odd, it has exactly $q^2$ solutions
in $\FF _q^3$, because the associated conic $Q\subset \PP^2(\FF _q)$ has cardinality $q+1$).
\end{proof}

\begin{lemma}\label{00ab1}
For every $k\in \FF _q$, $q$ odd, and any $a_1\in \FF _q^\ast$, $a_2\in \FF _q^\ast$ there are $x_1,x_2\in \FF _q$ such that $a_1x_1^2+a_2x_2^2 =k$.
\end{lemma}

\begin{proof}
If $k=0$, then take $x_1=x_2=0$. Now assume $k\ne 0$. The equation $a_1x_1^2+a_2x_2^2-kx_3^2 =0$ is the equation of a smooth conic $C\subset \PP^2(\FF _q)$, because
for odd $q$ and non-zero $a_1, a_2, k$ the partial derivatives of $a_1x_1^2+a_2x_2^2-kx_3^2$ have only $(0,0,0)$ as their common zero. We have $\sharp ({C}) =q+1$ (\cite[Part (i) of Theorem 5.2.6]{h})
and at most two of its points are contained in the line $L\subset \PP^2(\FF _q)$ with $x_3=0$ as its equation. If $(b_1:b_2:b_3)\in C\setminus C\cap L$,
then $b_3\ne 0$ and $a_1(b_1/b_3)^2+a_2(b_1/b_3)^2 =k$.
\end{proof}

The assumption ``~$q\equiv 1\pmod{4}$ if $n=2$~'' in the next result is necessary by part (i) of Proposition \ref{qq3}.

\begin{proposition}\label{00ab1+}
Assume $q$ odd. If $n=2$ assume $q\equiv 1\pmod{4}$. Let $M =(m_{ij})$ be an $n\times n$ matrix such that $m_{ij}+m_{ji} =0$ for all $i\ne j$, $m_{11}\ne m_{22}$
and $m_{ii} = m_{22}$ for all $i>2$. Then $\sharp (\Nm _0(M)_q) =(q+1)/2$ and $\Nm _0(M)_q\setminus \{0\}$ is the set of all $a\in \FF _q^\ast$
such that $-a/(m_{22}-m_{11})$ is a square. We have $0\in \Nm '_0(M)_q$ if and only if either $n\ge 4$ or $n=3$ and $q\equiv 1\pmod{4}$.
\end{proposition}

\begin{proof}
By Remark \ref{qq1} it is sufficient to do the case in which $M$ is a diagonal matrix. The case $n=2$ is true by part (iii2) of Proposition \ref{qq3}. Now assume $n\ge 3$. Taking the difference of (\ref{eqq2}) with (\ref{eqq1}) multiplied by $m_{11}$ we get
$(m_{22}-m_{11})(x_2^2+\cdots +x_n^2) =a$, while (\ref{eqq1}) gives $x_1^2 = -(x_2^2+\cdots +x_n^2)$. Thus if $-a/(m_{22}-m_{11})$ is not a square, then
$a\notin \Nm _0(M)_q$. If $-a/(m_{22}-m_{11})$ is a square, then we take $x_i=0$ for $i>3$, take $x_2$ and $x_3$ such that $(m_{22}-m_{11})(x_2^2+x_3^2) =a$
(Lemma \ref{00ab1+}) and then take $x_1$ with $x_1^2 = -a/(m_{22}-m_{11})$. Now take $a=0$. If $n\ge 4$ we take $x_1=0$, $x_j=0$ for all $j>4$
and find $(x_2,x_3,x_4) \in \FF _q^3\setminus \{(0,0,0)\}$ such that $x_2^2+x_2^2+x_3^2=0$ (take $x_3=1$ and use Lemma \ref{00ab1+} with $a_1=a_2=1$ and $k=-1$).
Now assume $a=0$ and $n=3$. We proved that we need to have $x_2^2+x_3^2 =0$ and hence we need to have $x_1=0$. There is $(x_2,x_3)\in \FF _q^2\setminus \{(0,0)\}$
with $x_2^2+x_3^2 =1$ if and only if $-1$ is a square in $\FF _q$, i.e. if and only if $q\equiv 1 \pmod{4}$.
\end{proof}

\begin{proposition}\label{00ab2}
Assume $q$ odd and $n\ge 3$. Let $M =(m_{ij})$ be an $n\times n$ matrix over $\FF _q$ such that $m_{ij}+m_{ji} =0$ for all $i\ne j$, and not all diagonal elements
are the same. Then $\sharp (\Nm _0(M) _q)\ge (q+1)/2$.
\end{proposition}

\begin{proof}
By Remark \ref{qq1} it is sufficient to the case in which $M$ is a diagonal matrix. If the diagonal entries of $M$ have only two different values, they we may rearrange
them so that they are $m_{11}$ and $m_{22}$ occurring at least twice. In this case we may apply the case $n=3$ of Proposition \ref{00ab1}.
Hence we may assume that $m_{11}$, $m_{22}$ and $m_{33}$ are different. It is sufficient to the case $n=3$. Take $x_1=1$ and then take $x_2,x_3$ such
that $x_3^2 -x_2^2 =-1$ (Lemma \ref{00ab1+}).
From (\ref{eqq2}) we get $(m_{33}+m_{22})x_2^2 = a-m_{11} +m_{33}$. If $m_{33}+m_{22}\ne 0$, then we get that all $a\in \FF _q$ such that $(a-m_{11})/(m_{33}+m_{22})$ are
squares in $\FF _q$ (and there are $(q+1)/2$ such elements, because products of squares are squares and a product of a non-zero square and a non-square is not a square) are contained in $\Nm _0(M)_q$. Similarly, we conclude if either $m_{11}+m_{22}\ne 0$ or $m_{11}+m_{33}\ne 0$. If $m_{11}+m_{22}=m_{11}+m_{33} =m_{22}+m_{33}=0$,
then $m_{11} =m_{22} =m_{33}=0$, because $q$ is odd. 
\end{proof}

\begin{lemma}\label{00ab4}
Let $r$ be a prime power. Let $f\in \FF _r[t_1,t_2]$ be a polynomial of degree at most $2$ with $f$ not a constant. Then $f$ assumes at least $\lceil  r/2\rceil$ values over $\FF _r$. 
\end{lemma}
\begin{proof}
Let $\phi : \FF _r^2\to \FF _r$ be the map induced by $f$. Since $\deg (f) \le 2 $ and $f$ is not constant, for each $a\in \FF _r$, $\phi ^{-1}(a)$ is an affine conic
and in particular
$\sharp (\phi ^{-1}(a)) \le 2r$. Hence $\sharp (\phi (\FF _r^2))\ge \lceil  r/2\rfloor$. 
\end{proof}

\begin{proposition}\label{00ab3}
Assume $q$ odd and $n\ge 3$. Let $M =(m_{ij})$ be an $n\times n$ matrix over $\FF _q$ such that there is $i\in \{1,\dots ,n\}$
with $m_{ij}+m_{ji} =0$ for all at least $2$ indices $j\ne i$ (say $j_1$ and $j_2$) and either $m_{j_1j_1}\ne m_{ii}$ or $m_{j_2j_2}\ne m_{ii}$ or $m_{j_1j_2}+m_{j_2j_1}\ne 0$.
Then $\sharp (\Nm _k(M)_q)) \ge (q+1)/2$ for all $k\in \FF _q$.
\end{proposition}

\begin{proof}
We reduce to the case $n=3$ and $m_{32}+m_{23} =m_{31}+m_{13} =0$ and either $m_{11} \ne m_{33}$ and $m_{22}\ne m_{33}$ or $m_{12}+m_{21} \ne 0$. By Remark \ref{qq1} we may assume that $m_{32}=m_{23} =m_{31}=m_{13} =0$. Taking the difference
between (\ref{eqq2}) and $m_{33}$ times (\ref{eqq2}) we get
$$(m_{11}-m_{33})x_1^2+(m_{12}+m_{21})x_1x_2 + (m_{22} -m_{33})x_2^2 = a -km_{33} .$$Apply Lemma \ref{00ab4}.
\end{proof}

\providecommand{\bysame}{\leavevmode\hbox to3em{\hrulefill}\thinspace}


\begin{thebibliography}{99}

\bibitem{b} E. Ballico, On the numerical range of matrices over a finite field, Linear Algebra Appl. (to appear).

\bibitem{bss} I. Blake, G. Seroussi and N. Smart, Elliptic Curves in Cryptography, London Math. Soc. Lect. Note Series 265, Cambridge University Press, Cambridge, 2000.

\bibitem{cjklr} J. I. Coons, J. Jenkins, D. Knowles, R. A. Luke and P. X. Rault, Numerical ranges over finite fields, Linear Algebra Appl. 501 (2016), 37--47.

\bibitem{gr} K. E. Gustafson and D. K. M. Rao, Numerical Range. The Field of Values of Linear Operators and Matrices, Springer, New York, 1997.


\bibitem{h} J. W. P. Hirschfeld, Projective geometries over finite fields, Clarendon Press, Oxford, 1979.

\bibitem{h1} J. W. P. Hirschfeld, Finite projective spaces of three dimensions, Oxford Mathematical Monographs, Oxford Science Publications, The Clarendon Press, Oxford University Press, New York, 1985.

\bibitem{ht} J. W. P. Hirschfeld and J. A. Thas,  General Galois geometries, Oxford Mathematical Monographs, Oxford Science Publications, The Clarendon Press, Oxford University Press, New York, 1991.

\bibitem{hj} R.A. Horn and C.R. Johnson, Matrix Analysis, Cambridge University Press, New York, 1985.


\bibitem{hj1} R.A. Horn and  C.R. Johnson, Topics in Matrix Analysis, Cambridge University Press, Cambridge, 1991.

\bibitem{ln} R. Lindl and H. Niederreiter, Introduction to finite fields and their applications, Cambridge University Press, Cambridge, 1994.

\bibitem{pt} P.J. Psarrakos and M.J. Tsatsomeros, Numerical range: (in) a matrix nutshell, Notes, National Technical University, Athens, Greece, 2004.


\bibitem{s} C. Small, Arithmetic of finite fields, Marcel \& Dekker 


\end{thebibliography}
\end{document}